\documentclass[12pt]{amsart}
\usepackage{microtype}
\usepackage{hyperref}
\usepackage{doi}
\usepackage[asymmetric,centering]{geometry}
\usepackage{amssymb}
\usepackage{amsmath}
\usepackage{amsfonts}
\usepackage{amsthm}
\usepackage{shuffle}
\usepackage{enumerate}

\newtheorem{thm}{Theorem}[section]
\newtheorem{prop}[thm]{Proposition}
\newtheorem{lem}[thm]{Lemma}
\newtheorem{cor}[thm]{Corollary}
\theoremstyle{definition}
\newtheorem*{rmk}{Remark}
\def\al{\alpha}

\def\de{\delta}

\def\zt{\zeta}

\def\op{\diamond}

\def\QS{\operatorname{QSym}}

\def\Q{\mathbb Q}
\def\Z{\mathbb Z}
\def\E{\mathcal E}

\newcommand\stone[2]{\genfrac{[}{]}{0pt}{}{#1}{#2}}
\newcommand\sttwo[2]{\genfrac{\{}{\}}{0pt}{}{#1}{#2}}

\makeatletter
\renewcommand{\sectionautorefname}{\S\@gobble}
\makeatother
\newcommand{\arxiv}[1]{arXiv:\href{https://arxiv.org/abs/#1}{#1}}

\begin{document}
\title{Truncated Multiple Zeta Values}

\author{Steven Charlton}
\address{}
\email{mail@stevencharlton.net}

\author{Michael E. Hoffman}
\address{Department of Mathematics, U. S. Naval Academy, Annapolis MD 21402 USA}
\email{meh@usna.edu}

\subjclass[2020]{Primary 11M32, Secondary 05E05}

\begin{abstract}
We generalize the definition of truncated multiple zeta values
by allowing arbitrary integers as arguments.  This leads to 
interesting identities, particularly with the argument 0.
Truncated multiple zeta values satisfy the same quasi-shuffle 
algebraic identities as multiple zeta values, but we need to
extend the algebra QSym of quasi-symmetric functions to a larger 
algebra.  Using this algebra, we are able to sum systematically
powers of harmonic and generalized harmonic numbers.  This leads
to summation identities such as
\[
\sum_{n=1}^\infty H_n^3\bigg(\zt(2)-\sum_{k=1}^n\frac1{k^2}-\frac1{n}\bigg)=
-\frac{11}2\zt(4)+\zt(3)+3\zt(2)-6.
\]
We also prove analogous identities involving alternating sums of
harmonic numbers and their powers.
\end{abstract}

\leavevmode\vspace{-1em}
\maketitle

\section{Introduction}
Let $n$ be a positive integer.
Truncated multiple zeta values are given by
\[
\zt_n(a_1,\dots,a_k)=\sum_{n\ge n_1>n_2>\dots>n_k\ge 1}\frac1{n_1^{a_1}n_2^{a_2}\cdots
  n_k^{a_k}}
\]
As noted in \S6.3 of \cite{HI}, this definition makes sense for any
positive integers $a_1,\dots,a_k$ since there are no convergence issues.
But we can go further:  the definition makes sense for {\it any} integers
$a_1,\dots,a_k$.  In particular, there are interesting properties when
0 appears in the exponent string.  We have
\begin{align*}
\zt_n(0,a_1,\dots,a_k)&=\sum_{n\ge n_0>n_1>\dots>n_k\ge 1}\frac1{n_1^{a_1}\cdots
    n_k^{a_k}}\\
&=\sum_{n\ge n_1>\dots>n_k\ge 1}\frac{n-n_1}{n_1^{a_1}\cdots n_k^{a_k}}\\
&=n\zt_n(a_1,\dots,a_k)-\zt_n(a_1-1,a_2,\dots,a_k)
\end{align*}
but also
\[
\zt_n(0,a_1,\dots,a_k)=\sum_{n\ge n_0>n_1>\dots>n_k\ge 1}\frac1{n_1^{a_1}\cdots n_k^{a_k}}
=\sum_{j=k}^{n-1}\zt_j(a_1,\dots,a_k) .
\]
In particular, noting that $\zt_n(0)=n$ and $\zt_n(1)=H_n$ ($n$th harmonic number),
\[
\zt_n(0,1)=n\zt_n(1)-\zt_n(0)=nH_n-n
\]
and so
\begin{equation}
\label{hsum}
\sum_{j=1}^{n-1} H_j=nH_n-n ,
\end{equation}
or, adding $H_n$ to both sides,
\begin{equation}
\label{hhsum}
\sum_{j=1}^n H_j=(n+1)H_n-n .
\end{equation}
Now the same quasi-shuffle product familiar from multiple zeta values
applies to truncated ones:  in particular,
\[
\zt_n(1)^2=\zt_n(2)+2\zt_n(1,1)
\]
and thus
\begin{align*}
\sum_{j=1}^{n-1}H_j^2=\zt_n(0,2)+2\zt_n(0,1,1)
&=n\zt_n(2)-\zt_n(1)+2n\zt_n(1,1)-2\zt_n(0,1)\\[-1ex]
& =n\zt_n(2)-\zt_n(1)+2n\zt_n(1,1)-2n\zt_n(1)+2n,
\end{align*}
or
\begin{equation}
\label{sumf}
\sum_{j=1}^nH_j^2=(n+1)H_n^2-(2n+1)H_n+2n .
\end{equation}
In order to have an algebraic foundation for this kind of calculation, we expand the algebra $\QS$ of quasi-symmetric functions to a larger quasi-shuffle algebra $\E$.  We discuss this in detail in \autoref{sec:qsym} below.  The algebraic results of 
\autoref{sec:qsym} are used in \autoref{sec:harmonic} to evaluate harmonic sums of the form
$\sum_{j=1}^n j^k\big(H_j^{(m)}\big)^p$, where 
$H_n^{(m)}=\sum_{i=1}^n i^{-m}$.
\par
Allowing negative integers in the exponent string means that sums of powers of integers are now truncated multiple zeta values, as in
\[
\zt_n(-1)=\sum_{j=1}^n j=\frac{n(n+1)}2.
\]
Then, e.g.,
\[
\sum_{j=1}^{n-1}\zt_j(-1)=\zt_n(0,-1)=n\zt_n(-1)-\zt_n(-2)
\]
or
\[
\sum_{j=2}^n\frac{(j-1)j}2=n\frac{n(n+1)}2-\zt_n(-2)
\]
from which follows
\[
\frac32\zt_n(-2)=\Big(n+\frac12\Big)\frac{n(n+1)}2
\]
and thus
\begin{equation}
\label{gsum}
\zt_n(-2)=\frac{n(n+1)(2n+1)}6 .
\end{equation}
In fact, we can prove Faulhaber's formula for $\zt_n(-k)$ as a polynomial
of degree $k+1$ in $n$.  We give further results on negative arguments
in \autoref{sec:neg} below.
\par
Summation results for harmonic sums can be used to prove various
results about infinite series.  If we let
\[
T_n(2)=\zt(2)-\zt_n(2)=\sum_{j=n+1}^\infty \frac1{j^2}
\]
be the $n$th ``tail'' of the series for $\zt(2)$, then Eq. \eqref{sumf}
above can be used to prove that (cf. \cite{He})
\[
\sum_{n=1}^\infty H_n^2\Big(T_n(2)-\frac1{n}\Big)=2-\zt(2)-2\zt(3) .
\]
In \autoref{sec:lim} we show how to prove many similar results.
\par
Multiple zeta values can be generalized to alternating or ``colored''
multiple zeta values such as
\[
\zt(\bar2,1)=\sum_{n_1>n_2\ge 1}\frac{(-1)^{n_1}}{n_1^2n_2} ,
\]
which have truncated counterparts.  In particular, $\zt_n(\bar1)$
is the truncation of the alternating harmonic series
$-1+\frac12-\frac13+\cdots$.  In \autoref{sec:alt} we obtain counterparts of the
results above with $H_n$ replaced by $\bar H_n=\zt_n(\bar1)$.
Although we prove results {\it ad hoc} in this section, a direction 
for future research is to define an appropriate quasi-shuffle algebra
extending $\E$.
\subsection*{Acknowledgements.}  SC is grateful to Max Planck Institute for Mathematics in
Bonn for support and hospitality during the preparation of this work.  
MEH thanks the organizers of the 17th Mathematical Society of Japan
Seasonal Institute for the opportunity to speak in Osaka and 
develop the ideas herein.
Views expressed in this paper are those of the authors
    and do not reflect the official policy of the U. S. Naval Academy, Department of the Navy, Department of War, or U. S. Government.
    
\section{The Extended Quasi-Symmetric Functions}\label{sec:qsym}
We recall the construction of \cite{HI}, where we start with a countable
set $A$ of ``letters'' that admits an operation $\op$ so that
$kA$, where $\Q\subset k$, is a commutative algebra.
Then if we take the noncommutative polynomial algebra $k\langle A\rangle$
generated by ``words'' in elements of $A$ (together with the empty word $\mathbf 1$),
we can give the vector space $k\langle A\rangle$ a commutative and
associative product $*$, defined as follows.  Put $w*\mathbf 1=\mathbf 1*w=w$ for all
words $w$, and set
\[
aw*bv=a(w*bv)+b(aw*v)+(a\op b)(w*v)
\]
for all letters $a,b$ and words $w,v$.  This is the quasi-shuffle
algebra on $(A,\op)$.
If the operation $\op$ is identically zero, then the quasi-shuffle
algebra on $(A,\op)$ is just the shuffle algebra on $k\langle A\rangle$.
If we take $A$ to be the positive integers and
$\op$ to be addition, then $(k\langle A\rangle,*)$ is the algebra
$\QS$ of quasi-symmetric functions.
\par
We will need a larger algebra $\E$, which is the quasi-shuffle algebra
on $(A,\op)$ with $A=\Z$ and $\op$ being addition.
We denote the letter corresponding to $i\in\Z$ by $z_i$, so the
operation $\op$ is $z_i\op z_j=z_{i+j}$.
Now $\E$ is graded, but unlike $\QS$ it has nonscalar elements
in all grades, including $z_0$ in grade 0.
(We remark that $\E$, like $\QS$, has a Hopf algebra structure
given by taking the comultiplication to be deconcatenation and
all the $z_i$ to be primitive.  While not explored here, the
Hopf algebra structure might be used as in \cite{H2005} to
understand various identities.)
\par
Our notation for Stirling numbers follows \cite{GKP}; we write
$\stone{n}{k}$ for an (unsigned) Stirling number of the first
kind (i.e., the number of decompositions of $\{1,\dots,n\}$ into
$k$ disjoint cycles), and $\sttwo{n}{k}$ for a Stirling number
of the second kind (i.e., the number of partitions of 
$\{1,\dots,n\}$ with $k$ blocks).  We also write 
$x^{\underline k}$ for the falling factorial 
$x(x-1)\cdots (x-k+1)$.
\begin{prop}
\label{stpwr}
For nonnegative integers $p$,
$z_0^{*p}=\sum_{k=0}^p\sttwo{p}{k}k!z_0^k$.
\end{prop}
\begin{proof}
Induct on $p$, the case of $p=0$ being interpreted as ${\mathbf 1} ={\mathbf 1}$.
Take the $*$-product of both sides of the conclusion
with $z_0$ to get
\begin{align*}
z_0^{*(p+1)}=\sum_{k=0}^p\sttwo{p}{k}k! \, \big((k+1)z_0^{k+1}+kz_0^k \big) &=
\sum_{k=1}^{p+1}\left(\sttwo{p}{k-1}+\sttwo{p}{k}k\right)k!\,z_0^k\\
& =\sum_{k=1}^{p+1}\sttwo{p+1}{k}k!\,z_0^k
\end{align*}
where we used the recurrence formula for Stirling numbers of the
second kind.
\end{proof}
\par
For each positive integer $n$, there is a homomorphism
$\zt_n:\E\to\Q$ given by $\zt_n(\mathbf 1)=1$ (recall $\mathbf 1$ is
the identity element of $\E$, i.e., the empty word) and
\[
\zt_n(z_{i_1}z_{i_2}\cdots z_{i_k})=\zt_n(i_1,\dots,i_k) .
\]
We note that $\zt_n(z_0^j)=\binom{n}{j}$.
For a string $I=(i_1,\dots,i_k)$ of integers we will sometimes write $z_I$
for $z_{i_1}\cdots z_{i_k}$.  This makes $\zt_n(z_I)=\zt_n(I)$.
\par
Define linear functions $H$ and $D$ from $\E$ to itself by $Hw=z_0w$
and $D(\mathbf 1)=0$, $D(z_{i_1}z_{i_2}\cdots z_{i_k})=z_{i_1-1}z_{i_2}\cdots z_{i_k}$.
These operators allow us to restate the two equations at the
beginning of the Introduction as
\begin{equation}
\label{Hw}
\zt_n(Hw)=n\zt_n(w)-\zt_n(Dw) 
\end{equation}
and
\begin{equation}
\label{Sw}
\zt_n(Hw)=\sum_{j=1}^{n-1}\zt_n(w) .
\end{equation}
Given a truncated multiple zeta value $\zt_n(w)$, we
can find $\sum_{j=1}^{n-1}\zt_j(w)$ as $\zt_n(Hw)$ by Eq. \eqref{Sw}.
In order to find $\sum_{j=1}^{n-1}j^p\zt_j(w)$,
we need a formula for $\zt_n(H(z_0^{*p}*w))$;
but in view of Proposition \ref{stpwr},
it suffices to have formulas for $\zt_n(H(z_0^j*w))$ for $j\le p$.
These are provided by Theorem \ref{H0w} below, but first we need
some lemmas.
\begin{lem}
\label{Dz0w}
For all words $w$, $\zt_n(D^i(z_0*w))=\zt_n(D^{i+1}w)+\zt_n(z_{-i}w)$.
\end{lem}
\begin{proof}
The result is immediate for $w=\mathbf 1$, so we can assume $w=z_{a_1}w'$.
Then
\[
z_0*w=z_0w+z_{a_1}(z_0*w')+z_{a_1}w'=z_0w+w+z_{a_1}(z_0*w') ,
\]
so
\[
D^i(z_0*w)=z_{-i}w+D^iw+z_{a_1-i}(z_0*w') .
\]
Compare with
\[
z_0*D^iw=z_0D^iw+z_{a_1-i}(z_0*w')+z_{a_1-i}w'=HD^iw+D^iw+
z_{a_1-i}(z_0*w')
\]
to get
\[
D^i(z_0*w)=z_0*D^iw-HD^iw+z_{-i}w
\]
and thus
\[
\zt_n(D^i(z_0*w))=n\zt_n(D^iw)-n\zt_n(D^iw)+\zt_n(D^{i+1}w)
+\zt_n(z_{-i}w)
\]
and the conclusion follows.
\end{proof}
For positive integers $k$, let $\zt_D(-k)$ be the linear combination of powers of $D$ obtained by replacing $n$ with $D$ in the polynomial $\zt_n(-k)$.
\begin{lem}
\label{crz}
For $i\ge 1$, $\zt_n(z_{-i}w)=\zt_n(-i)\zt_n(w)-\zt_n(\zt_D(-i)w)$.
\end{lem}
\begin{proof}
Writing $w=z_{a_1}z_{a_2}\cdots z_{a_l}$, the left-hand side is
\begin{align*}
\sum_{n\ge n_0>n_1>\dots>n_l\ge 1}
\frac{n_0^i}{n_1^{a_1}n_2^{a_2}\cdots n_l^{a_l}}&=
\sum_{n\ge n_1>\dots>n_l\ge 1}\frac{\zt_n(-i)-\zt_{n_1}(-i)}
{n_1^{a_1}n_2^{a_2}\cdots n_l^{a_l}}\\
&=\zt_n(-i)\zt_n(w)-\sum_{n\ge n_1>\dots>n_l\ge 1}
\frac{\zt_{n_1}(-i)}{n_1^{a_1}\cdots n_l^{a_l}} .
\end{align*}
If we let $c_{i,j}$ be the coefficient of $n^j$ in the polynomial $\zt_n(-i)$, then
\begin{align*}
\zt_n(z_{-i}w)&=\zt_n(-i)\zt_n(w)-\sum_{j=1}^{i+1}
c_{i,j}\zt_n(z_{a_1-j}z_{a_2}\cdots z_{a_l})\\
&=\zt_n(-i)\zt_n(w)-\sum_{j=1}^{i+1}c_{i,j}\zt_n(D^jw)
\end{align*}
and the conclusion follows.
\end{proof}
\begin{lem}
\label{faulco}
The coefficient of $n^j$ in the Faulhaber polynomial $\zt_{n-1}(-p)$ is
\[
\sum_{q=j}^{p+1}\frac{(-1)^{q+1-j}}{q+1}\sttwo{p}{q}\stone{q+1}{j}
\]
for $1\le j\le p+1$.
\end{lem}
\begin{proof}
Apply $\zt_i$ to Proposition \ref{stpwr} to get
\begin{equation}
\label{basic}
i^p=\sum_{q=1}^p\sttwo{p}{q}q!\,\binom{i}{q} ,
\end{equation}
which we can sum from $i=1$ to $i=n-1$ to obtain
\begin{align*}
\zt_{n-1}(-p)
&=\sum_{q=1}^p\sttwo{p}{q}q!\sum_{i=1}^{n-1}\binom{i}{q}
=\sum_{q=1}^p\sttwo{p}{q}q!\,\binom{n}{q+1}\\
&=\sum_{q=1}^p\sttwo{p}{q}\frac{n^{\underline{q+1}}}{q+1}=
\sum_{q=1}^p\sttwo{p}{q}\frac1{q+1}\sum_{j=1}^{q+1}(-1)^{q+1-j}
\stone{q+1}{j}n^j ,
\end{align*}
and the conclusion follows.
\end{proof}

 \begin{rmk} Writing the last equation of the preceding proof
as
\[
\zt_{n-1}(-p)=\sum_{q=2}^{p+1}\sttwo{p}{q-1}\frac1{q}n^{\underline q}
\]
and adding $n^p=\sum_{q=1}^p\sttwo{p}{q}n^{\underline q}$ gives
\[
\zt_n(-p)=\sum_{q=1}^{p+1}\left(\frac1{q}\sttwo{p}{q-1}+
\sttwo{p}{q}\right)n^{\underline q}=\sum_{q=1}^{p+1}\frac1{q}
\sttwo{p+1}{q}n^{\underline q}
\]
from which it follows that the coefficient of $n^j$ in 
$\zt_n(-p)$ is
\[
\sum_{q=1}^{p+1}\sttwo{p+1}{q}\stone{q}{j}\frac{(-1)^{q-j}}{q}.
\]
Since $\zt_n(-p)=\zt_{n-1}(-p)+n^p$, we have
\begin{equation}
\label{coef}
\sum_{q=j}^{p+1}\frac{(-1)^{q+1-j}}{q+1}\sttwo{p}{q}\stone{q+1}{j}
=\sum_{q=1}^{p+1}\sttwo{p+1}{q}\stone{q}{j}\frac{(-1)^{q-j}}{q}-\de_{j,p} .
\end{equation}
By the well-known formula for Faulhaber
polynomials $\zt_{n-1}(-p)$ (see the next section), both sides
of Eq. \eqref{coef} are 
\[
\frac1{p+1}\binom{p+1}{j}B_{p+1-j}
\]
where $B_n$ is the $n$th Bernoulli number.  Cf. \cite[p. 289 ff.]{GKP}.
\end{rmk}
\begin{lem}
\label{st2b}
For positive integers $n$ and $j$,
\[
\sum_{k=1}^j\sttwo{j}{k}k!\binom{n+1}{k+1}=\zt_n(-j) .
\]
\end{lem}
\begin{proof}
Replace $p$ with $j$ in Eq. \eqref{basic} and sum from $i=1$
to $i=n$.
\end{proof}
\begin{thm}
\label{H0w}
For all words $w$ and nonnegative integers $j$,
\[
\zt_n(H(z_0^j*w))=\binom{n}{j+1}\zt_n(w)
+\sum_{q=1}^{j+1}\frac{(-1)^{q+j}}{(j+1)!}\stone{j+1}{q}\zt_n(D^qw)\\
\]
\end{thm}
\begin{proof} Since
\[
\zt_n(H(z_0^j*w))=n\zt_n(z_0^j*w)-\zt_n(D(z_0^j*w))=
n\binom{n}{j}\zt_n(w)-\zt_n(D(z_0^j*w))
\]
it suffices to show that
\begin{equation}
\label{conc}
\zt_n(D(z_0^j*w))=j\binom{n+1}{j+1}
-\sum_{q=1}^{j+1}\frac{(-1)^{j+q}}{(j+1)!}\stone{j+1}{q}
\zt_n(D^qw) .
\end{equation}
In fact, one can show by induction on $p$ that
\begin{equation}
\label{inh}
\zt_n(D(z_0^{*p}*w))=\sum_{q=1}^p\sttwo{p}{q}q!q\binom{n+1}{q+1}\zt_n(w)
+\zt_n(\zt_D(-p)w)-\zt_n(D^p w),
\end{equation}
and this can be seen to be equivalent to \eqref{conc}.
(The equivalence is immediate for the coefficient of $\zt_n(w)$;
for the coefficients of $\zt_n(D^jw)$, $j>0$, one needs Lemma
\ref{faulco}.)
Establishing \eqref{inh} by induction requires Lemmas
\ref{Dz0w} and \ref{crz} and two identities.
\begin{enumerate}[(i)]
\item
If $\zt_n(-j)=\sum_{k=1}^{j+1}c_{j,k}n^k$, then
\[
n\sum_{k=1}^j\sttwo{j}{k}\binom{n+1}{k+1}+\sum_{k=1}^{j+1}c_{j,k}\zt_n(-k)
-\zt_n(-j)=\sum_{i=1}^{j+1}\sttwo{j+1}{i}i!\,i\binom{n+1}{i+1};
\]
\\[1ex]
\item \leavevmode\vspace{-2.5\baselineskip}
\[
\sum_{k=1}^{j+1}c_{j,k}[\zt_n(D^{k+1}w)+\zt_n(D^kw)-\zt_n(\zt_D(-k)w)]
=\zt_n(\zt_D(-j-1)w) .
\]
\end{enumerate}
Using Lemma \ref{st2b}, 
\begin{align}
\label{sns}
\sum_{k=1}^j\sttwo{j}{k}k!k\binom{n+1}{k+1} &=
\sum_{k=1}^j\sttwo{j}{k}(k+1)!\binom{n+1}{k+1}
-\sum_{k=1}^j\sttwo{j}{k}k!\binom{n+1}{k+1} \\
& =\sum_{k=1}^j\sttwo{j}{k}(n+1)^{\underline{k+1}} \ -
\zt_n(-j) \notag \\
&=
(n+1)\sum_{k=1}^j\sttwo{j}{k}n^{\underline k}\ -\zt_n(-j)
=n^j(n+1)-\zt_n(-j) \notag
\end{align}
and identity (i) is
\[
n^{j+1}(n+1)-n\zt_n(-j)+\sum_{k=1}^{j+1}c_{j,k}\zt_n(-k)-\zt_n(-j)=
n^{j+1}(n+1)-\zt_n(-j-1)
\]
or
\begin{equation}
\label{trv}
\zt_n(-j-1)=(n+1)\zt_n(-j)-\sum_{k=1}^{j+1}c_{j,k}\zt_n(-k) .
\end{equation}
Now
\begin{align*}
\sum_{k=1}^{j+1}c_{j,k}\zt_n(-k)&=\sum_{i=1}^n\sum_{k=1}^{j+1}
c_{j,k}i^k = \sum_{i=1}^n\zt_i(-k)\\
&=\sum_{i=1}^n[(n+1)-i]i^k
=(n+1)\zt_n(-k)-\zt_n(-k-1),
\end{align*}
and Eq. \eqref{trv} follows.  Identity (ii) says that $\zt_n(Rw)=0$,
where $R$ is the operator
\[
R=\sum_{k=1}^{j+1}c_{j,k}\zt_D(-k)-(D+1)\zt_D(-j)
+\zt_D(-j-1).
\]
After exchanging $n$ for $D$, we see that $R=0$ by Eq. \eqref{trv}.
\end{proof}
\par
We introduce another linear function $G:\E\to\E$ by setting
$G(\mathbf 1)=0$ and $G(z_{i_1}\cdots z_{i_k})=i_1z_{i_1}\cdots z_{i_k}$.
Evidently $GH=0$.
\begin{lem}
\label{comm}
$DG-GD=D$.
\end{lem}
\begin{proof}
Both sides annihilate scalars, so it suffices to consider the effect on
a nonempty word $w=z_{i_1}\cdots z_{i_k}$.  We have
\[
(DG-GD)w=i_1z_{i_1-1}\cdots z_{i_k}-(i_1-1)z_{i_1-1}\cdots z_{i_k}=
z_{i_1-1}\cdots z_{i_k}=Dw . \qedhere
\]
\end{proof}
The preceding result can be written $DG=(1+G)D$.  Since $1+G$ sends
$\mathbf 1$ to 1 and $z_{i_1}\cdots z_{i_k}$ to $(i_1+1)z_{i_1}\cdots z_{i_k}$,
$(1+G)^{-1}$ makes sense on any word not beginning with $z_{-1}$.
\begin{lem}
\label{DHl}
For $n\ge 2$, $DG(z_1^{*n})=n(H+1)z_1^{*(n-1)}$.
\end{lem}
\begin{proof}
We have
\[
z_1^{*n}=\sum_{I\vDash n}\binom{n}{I}z_I,
\]
where the sum is over all compositions $I$ of $n$.
Then for any $n\ge 2$ we can write
\begin{align*}
& DG(z_1^{*n}) \\
&=\sum_{(i_1,i_2,\dots,i_k)\vDash n}\binom{n}{i_1\ i_2\cdots i_k}
  i_1z_{i_1-1}z_{i_2}\cdots z_{i_k}\\
  & =\sum_{\substack{(i_2,\dots,i_k) \\\vDash n-1}}n\binom{n-1}{i_2\cdots i_k}z_0z_{i_2}\cdots z_{i_k}
  +\sum_{\substack{(i_1-1,i_2,\dots,i_k)\\\vDash n}}n\binom{n-1}{i_1-1\ i_2\cdots i_k}
  z_{i_1-1}z_{i_2}\cdots z_{i_k}\\[1ex]
  &=nHz_1^{*(n-1)}+nz_1^{*(n-1)}=n(H+1)z_1^{*(n-1)} \qedhere
\end{align*}
\end{proof}
Now define
\[
Z(t)=\sum_{n\ge 0}z_1^{*n}\frac{t^n}{n!}=\mathbf 1+tz_1+\frac{t^2}2z_1^{*2}+\cdots .
\]
Then by Lemma \ref{DHl}
\begin{align*}
DGZ(t)&=z_0t+\sum_{n\ge 2}n(H+1)z_1^{*(n-1)}\frac{t^n}{n!}\\
&=z_0t+t\sum_{n\ge 1}(H+1)z_1^{*n}\frac{t^n}{n!}=(H+1)tZ(t)-t\mathbf 1.
\end{align*}
Using Lemma \ref{comm}, this is
\begin{equation}
\label{gfd}
(1+G)DZ(t)=(H+1)tZ(t)-t\mathbf 1 .
\end{equation}
Hence
\begin{align*}
DZ(t)&=(1+G)^{-1}\big((H+1)tZ(t)-t\mathbf 1\big)\\
&=(1-G+G^2-G^3+\cdots)\big((H+1)tZ(t)-t\mathbf 1\big)\\
&=HtZ(t)+(1+G)^{-1}t(Z(t)-\mathbf 1) .
\end{align*}
Now apply $\zt_n$ to both sides:
\[
\zt_n(DZ(t))=nt\zt_n(Z(t))-t\zt_n(DZ(t))+t\zt_n((1+G)^{-1}(Z(t)-\mathbf 1))
\]
or
\[
\zt_n(DZ(t))=n\zt_n(Z(t))\frac{t}{1+t}+\zt_n((1+G)^{-1}(Z(t)-\mathbf 1))
\frac{t}{1+t} .
\]
Extract the coefficient of $\frac{t^m}{m!}$ to get the following.
\begin{prop}
\label{oneD}
For $m\ge 0$,
\begin{align*}
&\sum_{I\vDash m}\binom{m}{I}\zt_n(Dz_I) = {} \\
& -\sum_{j=1}^mn
\frac{(-1)^jm!}{(m-j)!}\zt_n(1)^{m-j}
-\sum_{j=1}^{m-1}\sum_{I\vDash m-j}\frac{(-1)^j}{i_1+1}\binom{m-j}{I}
\frac{m!}{(m-j)!}\zt_n(I) .
\end{align*}
\end{prop}
Now apply $D$ to both sides of Eq. \eqref{gfd} to get
\[
D(1+G)DZ(t)=z_{-1}Z(t)t+DZ(t)t
\]
or
\[
(2+G)D^2Z(t)=z_{-1}Z(t)t+HZ(t)t^2+(1+G)^{-1}(Z(t)-\mathbf 1)t^2
\]
from which follows
\[
D^2Z(t)=z_{-1}Z(t)t+\frac12HZ(t)t^2+(2+G)^{-1}(1+G)^{-1}(Z(t)-\mathbf 1)t^2 .
\]
Apply $\zt_n$ to both sides:
\begin{align*}
\zt_n(D^2Z(t))= {}  \zt_n(-1)&\zt_n(Z(t)t) -\zt_n(\zt_D(-1)Z(t)t)
+\frac{n}2\zt_n(Z(t)t^2)\\
&  {} -\frac12\zt_n(DZ(t)t^2)+\zt_n\big((2+G)^{-1}(1+G)^{-1}(Z(t)-\mathbf 1)t^2 \big) ,
\end{align*}
which after some manipulation is
\begin{align*}
& \Big(1+\frac{t}2\Big)\zt_n(D^2Z(t)) \\
& = \zt_n(-1)\zt_n(Z(t))t
+\zt_n\big( ((1+G)^{-1}-2(2+G)^{-1})(Z(t)-\mathbf 1) \big)\frac{t^2}2 .
\end{align*}
Multiply both sides by $(1+\frac{t}2)^{-1}$ and extract the coefficient
of $\frac{t^m}{m!}$ to get the following result.
\begin{prop}
\label{twoD}
For $m\ge 0$,
\begin{align*}
\sum_{I\vDash m}\binom{m}{I}\zt_n(D^2z_I)=
& -\sum_{j=1}^m\Big({-}\frac12\Big)^j n(n+1)\frac{m!}{(m-j)!}
\zt_n(1)^{m-j}\\
& {} -2\sum_{j=2}^{m-1}\sum_{I \vDash m-j}\Big({-}\frac12\Big)^j\frac{i_1}{(i_1+1)(i_1+2)}\cdot
\frac{m!}{(m-j)!}\binom{m-j}{I}\zt_n(I) .
\end{align*}
\end{prop}
\section{Negative Arguments}\label{sec:neg}
Note that
\begin{align*}
\zt_n(a_1,\dots,a_l,-1)&=\sum_{n\ge n_1>\dots>n_l>n_{l+1}\ge1}\frac{n_{l+1}}{n_1^{a_1}\cdots
  n_l^{a_l}}\\
&=\sum_{n\ge n_1>\dots>n_l>1}\frac{\zt_{n_l-1}(-1)}{n_1^{a_1}\cdots n_l^{a_l}}\\
&=\frac12\sum_{n\ge n_1>\dots>n_l\ge1}\frac{n_l(n_l-1)}{n_1^{a_1}\cdots n_l^{a_l}}\\
&=\frac12\zt_n(a_1,\dots,a_{l-1},a_l-2)-\frac12\zt_n(a_1,\dots,a_{l-1},a_l-1) .
\end{align*}
If we let $D_t$ act on words from the right rather than the left
(so, e.g., $D_tz_az_b=z_az_{b-1}$), then we can state the last line
above as $\zt_n(\zt_{D_t-1}(-1)w)$ for $w=z_{a_1}z_{a_2}\cdots z_{a_l}$.
More generally, we have the following analogue of Lemma \ref{crz}.
\begin{prop}
\label{crzr}
For $i\ge1$, $\zt_n(wz_{-i})=\zt_n(\zt_{D_t-1}(-i)w)$.
\end{prop}

Faulhaber's formula
\[
\zt_{n-1}(-k)=\frac1{k+1}\sum_{j=0}^k\binom{k+1}{j}B_jn^{k+1-j}
\]
expresses $\zt_{n-1}(-k)$ as a polynomial in $n$ with rational 
coefficients.
We give a proof of this formula, which we state in
terms of Bernoulli polynomials.
\begin{thm} For $k\ge 1$,
\[
\zt_n(-k)=\frac{B_{k+1}(n+1)-B_{k+1}}{k+1} .
\]
\end{thm}
\begin{proof}
We use induction on $k$.  Since $B_1(t)=t-\frac12$, 
$B_2(t)=t^2-t+\frac16$, and $B_3(t)=t^3-\frac32t^2+\frac12t$,
we have already established the result for $k\le 2$.
Now assume the result for $k\ge 1$ and consider
\[
\zt_n(0,-k)+\zt_n(-k)=(n+1)\zt_n(-k)-\zt_n(-k-1).
\]
The left-hand side is $\sum_{j=1}^n\zt_j(-k)$, so
\[
\zt_n(-k-1)=(n+1)\zt_n(-k)-\sum_{j=1}^n\zt_j(-k) .
\]
Thus, to prove the conclusion it suffices to show
\[
\frac{B_{k+2}(n+1)-B_{k+2}}{k+2}=\frac{n+1}{k+1}(B_{k+1}(n+1)-B_{k+1})
  -\sum_{j=1}^n\frac{B_{k+1}(j+1)-B_{k+1}}{k+1} .
\]
We can assume $k>0$, and since $B_k(1)-B_k=0$ for $k\ne 1$, the
preceding equation can be replaced by
\begin{align*}
\frac{B_{k+2}(n+1)-B_{k+2}}{k+2}&=\frac{n+1}{k+1}(B_{k+1}(n+1)-B_{k+1})
-\sum_{j=0}^n\frac{B_{k+1}(j+1)-B_{k+1}}{k+1}\\
&=\frac{(n+1)B_{k+1}(n+1)}{k+1}-\sum_{j=0}^n\frac{B_{k+1}(j+1)}{k+1}\\
&=\frac{nB_{k+1}(n+1)}{k+1}-\sum_{j=0}^{n-1}\frac{B_{k+1}(j+1)}{k+1} .
\end{align*}
Now use the property $B_m(x+1)=mx^{m-1}+B_m(x)$ repeatedly to write
\begin{align*}
B_{k+1}(n+1)&=(k+1)n^k+B_{k+1}(n)\\
&=(k+1)n^k+(k+1)(n-1)^k+B_{k+1}(n-1)\\
&=(k+1)n^k+(k+1)(n-1)^k+(k+1)(n-2)^k+B_{k+1}(n-2)
\end{align*}
and so forth, and after cancellation the right-hand side above is
\[
(n)(n^k)+(n-1)(n-1)^k+\cdots=\zt_n(-k-1) .
\]
A similar argument shows that the left-hand side is also $\zt_n(-k-1)$.
\end{proof}
\par
If $I=(i_1,\dots,i_l)$ is a string of negative integers, let $\ell(I)=l$ 
and $|I|=-\sum_{j=1}^l i_j$.
\begin{prop}\label{znI} For a string of negative integers $I$, $\zt_n(I)$ is a polynomial
in $n$ of degree $|I|+\ell(I)$.
\end{prop}
\begin{proof} Induction on $\ell(I)$, starting with Faulhaber's formula.
For the induction step, use the relation
\[
\zt_n(a_1,\dots,a_l,-k)=
\frac1{k+1}\sum_{j=0}^{k}B_j\binom{k+1}{j}\zt_n(a_1,\dots,a_l-k-1+j)
\]
that follows from Proposition \ref{crzr} and Faulhaber's formula.
\end{proof} 
We now give some results on repeated arguments.  Recall our notational
convention, e.g., $\zt_n(-1,-1)=\zt_n(z_{-1}^2)$.
\begin{prop}
\label{allone}
For $k\ge 1$,
\[
\zt_n(z_{-1}^k)=\stone{n+1}{n-k+1} .
\]
\end{prop}
\begin{proof}
This is immediate from consideration of the generating function
for Stirling numbers of the first kind, i.e.,
\begin{equation}
\label{stgf}
x(x+1)\cdots(x+n-1)(x+n)=\sum_{j=1}^{n+1}\stone{n+1}{j}x^j,
\end{equation}
since the coefficient of $x^{n-k+1}$ is evidently the sum of $k$-fold
products from the set $\{1,2,\dots,n\}$.
\end{proof}
There is a simple relation between $\zt_n$ of positive and negative
repeated arguments.
\begin{prop} For nonnegative integers $a$, $\zt_n(z_{-a}^k)=
(n!)^a\zt_n(z_a^{n-k})$.
\end{prop}
\begin{proof} Immediate from definitions.
\end{proof}
We call this the reciprocity relation with exponent $a$.
For $a=0$ it is just symmetry of binomial coefficients, and for
$a=1$ it allows us to conclude from Proposition \ref{allone}
that 
\[
\zt_n(z_1^k)=\frac1{n!}\stone{n+1}{k+1}.
\]
\begin{prop}
For $k\ge 1$,
\[
\zt_n(z_{-2}^k)=\sum_{j=-k}^k(-1)^{j+k}\stone{n+1}{n+1-k+j}\stone{n+1}{n+1-k-j} .
\]
\end{prop}
\begin{proof} Multiply Eq. \eqref{stgf} by the same equation with $x$
replaced by $-x$ and examine coefficients.
\end{proof}
From \cite{OEIS} we have that
\[
\zt_n(z_{-2}^k)=\text{coefficient of $\frac{z^{2n+2}t^k}{(2n+2)!}$
in $\cosh\left(\frac2{\sqrt{t}}\sin^{-1}\left(\frac{z\sqrt{t}}2\right)\right)$}.
\]
From reciprocity,
\[
\zt_n(z_2^k)=\frac1{(n!)^2}\zt_n(z_{-2}^{n-k}) .
\]
\section{Harmonic Sums}\label{sec:harmonic}
We start by obtaining a formula for the sum over $k$ of $k^a$ times
a power of the generalized harmonic numbers $H_k^{(r)}=\zt_k(r)$.
Recall that for \( I = (i_1,\ldots,i_k) \) the expression \( z_I \)
denotes \( z_{i_1} \cdots z_{i_k} \).
\begin{thm}\label{H0jr}
For $r\ge 1$, 
\[
\sum_{k=1}^n k^a(H_k^{(r)})^p = (H_n^{(r)})^p(\de_{a,0} + \zt_n(-a)) 
-\sum_{j=0}^a \frac{1}{j+1} \binom{a}{j} B_{a-j}\sum_{I\vDash p}
\binom{p}{I}\zt_n(D^{j+1}z_{rI}) ,
\]
where for a composition $I=(i_1,\dots,i_d)$ we write $rI$ for $(ri_1,\dots,ri_d)$.  
\end{thm}
\begin{proof}
By properties of the stuffle product
\[
\zt_k(r)^p=\sum_{I\vDash p}\binom{p}{I}\zt_k(z_{rI}) .
\]
Then
\[
\zt_n(z_0^{*a}*z_r^{*p})=\zt_n(0)^a\zt_n(r)^p=\sum_{I\vDash p}
\binom{p}{I}\zt_n(z_0^{*a}*z_{rI});
\]
the left-hand side of the conclusion is
\[
n^a(H_n^{(r)})^p+\sum_{k=1}^{n-1}\zt_k(0)^a\zt_k(r)^p .
\]
By Eq. \eqref{Sw} we have
\[
\sum_{k=1}^{n-1}\zt_k(0)^a\zt_k(r)^p=
\sum_{I\vDash p}\sum_{k=1}^{n-1}\binom{p}{I}\zt_k(z_0^{*a}*z_{rI})=
\sum_{I\vDash p}\binom{p}{I}\zt_n(H(z_0^{*a}*z_{rI})) ,
\]
which by Eq. \eqref{Hw} is
\[
\sum_{I\vDash p} \binom{p}{I}(n\zt_n(z_0^{*a}*z_{rI})-\zt_n(D(z_0^{*a}*z_{rI})))=
n^{a+1}(H_n^{(r)})^p-\sum_{I\vDash p}\binom{p}{I}\zt_n(D(z_0^{*a}*z_{rI})).
\]
Now use Eqs. \eqref{inh})and \eqref{sns} to write this as
\begin{align*}
n^{a+1}(H_n^{(r)})^p-\bigg[(n^a(n+1) & {} -\zt_n(-a))(H_n^{(r)})^p \\[-1ex]
& {}- \sum_{I\vDash p}\binom{p}{I}(\zt_n(\zt_D(-a)z_{rI})-\zt_n(D^az_{rI}))\bigg],
\end{align*}
which can be simplified to obtain the result.
\end{proof}
\begin{cor}
\label{Hsum}
For positive integers $p$,
\begin{align*}
\sum_{k=1}^nH_k^p= (n+1)H_n^p{} & +(-1)^p p!\, n 
+\sum_{j=1}^{p-1}\frac{(-1)^jp!}{2(p-j)!}(2n+1)H_n^{p-j}
\\
& -\sum_{j=1}^{p-1}\frac{(-1)^jp!}{2(p-j)!}\sum_{I\vDash p-j}
\frac{i_1-1}{i_1+1}\binom{p-j}I\zt_n(I).
\end{align*}
\end{cor}
\begin{proof}
Set $a=0$ and $r=1$ in the preceding result to get
\[
\sum_{k=1}^n H_k^p=(n+1)H_n^p-\sum_{I\vDash p}\binom{p}{I}\zt_n(Dz_I).
\]
Then applying Proposition \ref{oneD} gives
\begin{align*}
 \sum_{k=1}^nH_k^p  = {} (n+1)H_n^p & {} +
\sum_{j=1}^p\frac{(-1)^j\,p!\,n}{(p-j)!}
H_n^{p-j}\\
& {} +\sum_{j=1}^{p-1}\frac{(-1)^jp!}{(p-j)!}\sum_{I\vDash p-j}
\frac1{i_1+1}\binom{p-j}{I}\zt_n(I),
\end{align*}
which can be rearranged into the conclusion.
\end{proof}
The case $p=3$ of the preceding result is
\[
\sum_{k=1}^n H_k^3=(n+1)H_n^3-\frac32(2n+1)H_n^2 +\frac12\zt_n(2)
+3(2n+1)H_n-6n
\]
a result that appears in one of Ramanujan's notebooks \cite{B},
and can also be found in Spie\ss\ \cite{S}.
\par
The following result generalizes Eq. \eqref{hsum} of the Introduction.
\begin{prop}
For nonnegative integers $j$,
\[
\sum_{k=1}^{n-1}\binom{k}{j}H_k=\binom{n}{j+1}\left(H_n-\frac1{j+1}\right) .
\]
\end{prop}
\begin{proof} Set $w=z_1$ in Theorem \ref{H0w} to get
\begin{align*}
\sum_{k=1}^{n-1}\binom{k}{j}H_k&=\binom{n}{j+1}H_n+
\sum_{q=1}^{j+1}\frac{(-1)^{q+j}}{(j+1)!}\stone{j+1}{q}\zt_n(1-q)\\
&=\binom{n}{j+1}H_n+\frac{(-1)^{q+j}}{(j+1)!}\left(j\stone{j}{q}+\stone{j}{q-1}\right)\sum_{i=1}^ni^{q-1}\\
&=\binom{n}{j+1}H_n+\sum_{i=1}^n\frac{j}{i}\sum_{q=1}^j\frac{(-1)^{j-q}}{(j+1)!}\stone{j}{q}i^q
-\sum_{i=1}^n\sum_{q=1}^{j+1}\frac{(-1)^{j-q+1}}{(j+1)!}
\stone{j}{q-1}i^{q-1}\\
&=\binom{n}{j+1}H_n+\frac1{j+1}\sum_{i=1}^n\left[\binom{i-1}{j-1}-\binom{i}{j}\right]\\
&=\binom{n}{j+1}H_n-\frac1{j+1}\sum_{i=1}^n\binom{i-1}{j} \qedhere
\end{align*}
and the conclusion follows.
\end{proof}
We can use the preceding result to write $\sum_{k=1}^nk^jH_k$ in
terms of Faulhaber polynomials.
\begin{thm}
\label{pwrhk}
For nonnegative integers $j$,
\[
\sum_{k=1}^n k^jH_k=\zt_n(-j)H_{n+1}-\sum_{i=1}^n\frac{\zt_i(-j)}{i+1} .
\]
\end{thm}
\begin{proof}
If $j=0$ the conclusion reads $\sum_{k=1}^nH_k=nH_{n+1}-\sum_{i=1}^n\frac{i}{i+1}$,
which can be seen to agree with Eq. \eqref{hhsum}.
So we can assume $j\ge 1$.  Then  
\begin{align*}
\sum_{k=1}^n k^jH_k &=\sum_{k=1}^n\sum_{i=1}^j\sttwo{j}{i}\binom{k}{i}\,i!\,H_k
\,=\,\sum_{i=1}^j\sttwo{j}{i}\,i!\,\sum_{k=1}^n\binom{k}{i}H_k\\
&=\sum_{i=1}^j\sttwo{j}{i}\,i!\,\binom{n+1}{i+1}\left(H_{n+1}-\frac1{i+1}\right)
\\
&=H_{n+1}\sum_{i=1}^j\sttwo{j}{i}\,i!\,\binom{n+1}{i+1}-
\sum_{i=1}^j\sttwo{j}{i}\frac{i!}{i+1}\binom{n+1}{i+1}.
\end{align*}
The coefficient of $H_{n+1}$ is $\zt_n(-j)$ by Lemma \ref{st2b}.
If we let $P_j(n)=\sum_{i=1}^j\sttwo{j}{i}\frac{i!}{i+1}\binom{n+1}{i+1}$, then
\[
\frac{\zt_n(-j)}{n+1}+P_j(n-1)=
\sum_{i=1}^j\sttwo{j}{i}\frac{i!}{i+1}\left[\binom{n}{i}+
\binom{n}{i+1}\right]=P_j(n)
\]
and therefore (since $P_j(0)=0$)
\[
P_j(n)=\sum_{i=1}^j\frac{\zt_i(-j)}{i+1} .
\]
The conclusion follows.
\end{proof}
The first few cases of Theorem \ref{pwrhk} are as follows.
\begin{align}
\sum_{k=1}^n kH_k&=\zt_n(-1)H_{n+1}-\frac12\binom{n+1}{2}\\
\label{sqr}  
\sum_{k=1}^n k^2H_k&=\zt_n(-2)H_{n+1}-\frac{4n+5}{18}\binom{n+1}2\\
\sum_{k=1}^n k^3H_k&=\zt_n(-3)H_{n+1}-\frac{3n+1}8\binom{n+2}3 .
\end{align}
Here is another corollary of Theorem \ref{H0jr}.
\begin{cor}
\label{kH}
For positive integers $n$ and $p$,
\begin{multline*}
\sum_{k=1}^n kH_k^p=\zt_n(-1) H_n^p
+\frac12\sum_{j=1}^p\frac{(-1)^jp!}{(p-j)!}\,n\,H_n^{p-j}(2^{-j}(n+1)-1)+\\
\sum_{j=2}^{p-1}\sum_{I\vDash p-j}\frac{(-1)^j\,p!}{(p-j)!}
\binom{p-j}{I}\frac{\zt_n(I)}{i_1+1}\left(\frac{i_1}{2^j(i_1+2)}-\frac12\right)
+\frac{p}2\sum_{I\vDash p-1} \binom{p-1}{I}\frac{\zt_n(I)}{i_1+1} .
\end{multline*}
\end{cor}
\begin{proof}
Use Theorem \ref{H0jr} with $a=1$ and then apply Propositions
\ref{oneD} and \ref{twoD}.  
\end{proof}
The cases $p=2$ and $p=3$ of Corollary \ref{kH} are respectively
\begin{align}
\label{kH2}
\sum_{k=1}^n kH_k^2 = {} \zt_n(-1)H_n^2 {} & -\left(\!\binom{n}2-\frac12\right)H_n
+\frac{n(n-3)}4
\intertext{and}
\label{kH3}
\sum_{k=1}^nkH_k^3= {} \zt_n(-1) H_n^3 {} & -\frac32\left(\!\binom{n}2-\frac12\right)H_n^2 \\
& {} +\frac{3n^2-9n-5}4H_n-\frac14H_n^{(2)}
-\frac38n(n-7). \notag
\end{align}
Cf. \cite{JS}.
We can also use Theorem \ref{H0jr} to get formulas for sums of the
form $\sum_{k=1}^n k^2H_k^p$.
\begin{prop}
\label{k2H}
For positive integers $n$,  
\[
\sum_{k=1}^n k^2H_k^2=\zt_n(-2) H_n^2-
\frac{4n^3-3n^2-n+3}{18} H_n+\frac{n(8n^2-15n+25)}{108} .
\]
\end{prop}
\section{Limit Theorems}\label{sec:lim}
In this section we show how the results of the previous section can
be used to obtain formulas for various series, as discussed in the
Introduction.  To evaluate these limits, two basic results are needed.
\begin{prop}
\label{fti}
If $a_0>1$, then
$\sum_{k=1}^\infty\frac{\zt_k(a_1,\dots,a_l)}{(k+1)^{a_0}}=\zt(a_0,a_1,\dots,a_l)$.
\end{prop}
\begin{proof}
Immediate from definitions.
\end{proof}  
\begin{prop}
\label{fs}
If $a_1>0$, then
$\sum_{k=1}^\infty\frac{\zt_k(a_1,\dots,a_l)}{k(k+1)}=\zt(a_1+1,a_2,\dots,a_l)$.
\end{prop}
\begin{proof}
Since $\frac1{k(k+1)}=\frac1{k}-\frac1{k+1}$,
\begin{align*}
\sum_{k=1}^\infty\frac{\zt_k(a_1,\dots,a_l)}{k(k+1)}
&=\frac1{l}\zt_l(a_1,\dots,a_l)+\sum_{j=1}^\infty\frac1{l+j}(\zt_{l+j}(a_1,\dots,a_l)
-\zt_{l+j-1}(a_1,\dots,a_l))\\
&=\frac1{l}\zt_l(a_1,\dots,a_l)+\sum_{j=1}^\infty\frac{\zt_{l+j-1}(a_2,\dots,a_l)}
{(l+j)^{a_1+1}}\\
&=\sum_{j=0}^\infty\frac{\zt_{l+j-1}(a_2,\dots,a_l)}{(l+j)^{a_1+1}}
\end{align*}
and the conclusion follows by the preceding result.
\end{proof}
Now let $T_n(2)=\zt(2)-\zt_n(2)=\sum_{j=n+1}^\infty\frac1{j^2}$.
Since $T_n(2)-\frac1{n}\sim \frac1{2n^2}$, it follows from the integral test
that
\[
\sum_{n=1}^\infty H_n^p\left(T_n(2)-\frac1{n}\right)
\]
converges for all $p\ge 1$.  Also, since
\begin{equation}
\label{tn2}
T_n(2)-\frac1{n}=-\sum_{j=n}^\infty \frac1{j(j+1)^2}
\end{equation}
we have
\[
\sum_{n=1}^\infty H_n^p\left(T_n(2)-\frac1{n}\right)=
-\sum_{n=1}^\infty\sum_{j=n}^\infty\frac{H_n^p}{j(j+1)^2}=
-\sum_{j=1}^\infty\sum_{n=1}^j\frac{H_n^p}{j(j+1)^2},
\]
which by Corollary \ref{Hsum} can be written
\begin{multline*}
-\sum_{j=1}^\infty\frac{H_j^p}{j(j+1)}-(-1)^pp!\sum_{j=1}^\infty\frac1{(j+1)^2}
-\sum_{j=1}^\infty\sum_{k=1}^{p-1}\frac{(-1)^kp!}{2(p-k)!}H_j^{p-k}
\frac{2j+1}{j(j+1)^2}\\
-\sum_{j=1}^\infty\sum_{k=1}^{p-1}\sum_{I\vDash p-k}\frac{(-1)^kp!}{2(p-k)!}
\frac{i_1-1}{i_1+1}\binom{p-k}{I}\frac{\zt_j(I)}{j(j+1)^2}
\end{multline*}
or
\begin{multline*}
  -(-1)^pp!(\zt(2)-1)-\sum_{I\vDash p}\binom{p}{I}
  \sum_{j=1}^\infty\frac{\zt_j(I)}{j(j+1)}\\
-\sum_{k=1}^{p-1}\frac{(-1)^kp!}{(p-k)!}\sum_{I\vDash p-k}\binom{p-k}{I}
\sum_{j=1}^\infty\zt_j(I)\left[\frac{i_1}{i_1+1}\cdot\frac1{j(j+1)}
+\frac1{i_1+1}\cdot\frac1{(j+1)^2}\right]
\end{multline*}
Then Propositions \ref{fti} and \ref{fs} can be applied to evaluate each
sum on $j$, giving the following result.
\begin{prop}
For $p\ge 1$,
\label{sta}  
\begin{multline*}
\sum_{n=1}^\infty H_n^p \Big(T_n(2) - \frac{1}{n}\Big) = 
 - (-1)^p p! (\zt(2) - 1) - \sum_{I \vDash p} \binom{p}{I} \zt(I^+) \\
    - \sum_{k=1}^{p-1} \sum_{I \vDash k} \frac{(-1)^{p-k} p!}{(i_1+1) k!} \binom{k}{I} \zt(I^+)  - \sum_{k=1}^{p-1} \sum_{I \vDash k} \frac{(-1)^{p-k} p!}{(i_1+1) k!} \binom{k}{I} i_1 \zt(2,I) ,
\end{multline*}
where for a composition $I=(i_1,i_2\dots,i_k)$, $I^+$ means
$(i_1+1,i_2,\dots,i_k)$.
\end{prop}
In the case $p=1$, the right-hand side reduces to $\zt(2)-1-\zt(2)=-1$.
The next few cases are as follows.
{\allowdisplaybreaks
\begin{align*}
\sum_{n=1}^\infty H_n^2\left(T_n(2)-\frac1{n}\right) = {} & -2\zt(3)-\zt(2)+2\\
\sum_{n=1}^\infty H_n^3\left(T_n(2)-\frac1{n}\right) = {} & -\frac{11}2\zt(4)+\zt(3)
+3\zt(2)-6\\
\sum_{n=1}^\infty H_n^4\left(T_n(2)-\frac1{n}\right) = {} & -\frac{43}2\zt(5)
+\frac12\zt(4)-4\zt(3)-12\zt(2)+24\\
\sum_{n=1}^\infty H_n^5\left(T_n(2)-\frac1{n}\right) = {} & -\frac{349}4\zt(6)
-14\zt(3)^2+\frac{57}2\zt(5)-15\zt(2)\zt(3)\\*
&\quad{}-\frac52\zt(4)+20\zt(3)+60\zt(2)-120
\end{align*}
}
\begin{rmk}
From consideration of Proposition \ref{sta} it is evident that
when $p>4$ the portion of $\sum_{n=1}^\infty H_n^p(T_n(2)-n^{-1})$ of
weight $\le 4$ is
\[
(-1)^pp!\left(1-\frac12\zt(2)-\frac1{6}\zt(3)+\frac1{48}\zt(4)\right) .
\]
\end{rmk}
If in the preceding result
$H_n$ is replaced by $H_n^{(r)}$ with $r\ge 2$, the corresponding
limit theorem is considerably simpler.
\begin{prop}\label{prop:hrnp}
If $r\ge 2$, then
\[
\sum_{n=1}^\infty \left( H_n^{(r)} \right)^p \left(T_n(2)-\frac1{n}\right)=
\sum_{I\vDash p}\binom{p}{I}(-\zt((rI)^+)+\zt(rI)-\zt(2,(rI)^-)) 
\]
where for a composition $I=(i_1,i_2\dots,i_k)$, $I^-$ means
$(i_1-1,i_2,\dots,i_k)$.
\end{prop}
\begin{proof}
First note that if $I=(i_1,\dots,i_j)$ with $i_1>0$, then
\begin{equation}
\label{firstl}
\sum_{k=1}^\infty\frac{\zt_n(I)}{k(k+1)}=\zt(I^+)
\end{equation}
by Proposition \ref{fs}, and
\begin{equation}
\label{secondl}
\sum_{k=1}^\infty\frac{\zt_k(I)}{k(k+1)^2}=\zt(I^+)-\zt(2,I)
\end{equation}
using $\frac1{k(k+1)^2}=\frac1{k(k+1)}-\frac1{(k+1)^2}$ and
Propositions \ref{fti} and \ref{fs}.
From Eq. \eqref{tn2} and Theorem \ref{H0jr} we have
\begin{align*}
\sum_{n=1}^\infty(H_n^{(r)})^p\left(T_n(2)-\frac1{n}\right)
&=-\sum_{n=1}^\infty\sum_{k=1}^n \frac{(H_k^{(r)})^p}{n(n+1)^2}\\
&=-\sum_{n=1}^\infty\frac{(H_n^{(r)})^p}{n(n+1)}+
\sum_{n=1}^\infty\frac1{n(n+1)^2}\sum_{I\vDash p}\binom{p}{I}\zt_n(Dz_{rI})\\
&=-\sum_{n=1}^\infty\sum_{I\vDash p}\binom{p}{I}\frac{\zt_n(rI)}{n(n+1)}
+\sum_{n=1}^\infty\sum_{I\vDash p}\binom{p}{I}\frac{\zt_n((rI)^-)}{n(n+1)^2}
\end{align*}
and the conclusion follows using Eqs. \eqref{firstl} and \eqref{secondl}.
\end{proof}
For example,
\begin{align*}
\sum_{n=1}^\infty H_n^{(2)}\left(T_n(2)-\frac1{n}\right)&=-2\zt(3)+\zt(2)\\
\sum_{n=1}^\infty (H_n^{(2)})^2\left(T_n(2)-\frac1{n}\right)&=
-\frac72\zt(5)+\frac52\zt(4). 
\end{align*}
Using the asymptotic series
\begin{equation}
\label{ems2}
T_n(2)\sim \frac1{n}-\frac1{2n^2}+\sum_{i\ge 1}\frac{B_{2i}}{n^{2i+1}}
\end{equation}
coming from the Euler-Maclaurin summation formula, one sees that
sums of the form
\begin{equation}
\label{EMs}
\sum_{n=1}^\infty nH_n^p\left(T_n(2)-\frac1{n}+\frac1{2n^2}\right) 
\end{equation}
converge for $p\ge 1$.  While one can use Corollary \ref{kH} to get a
general formula for \eqref{EMs}, here we give just the results following
from Eqs. \eqref{kH2} and \eqref{kH3}.
\begin{prop}
\begin{align*}
\sum_{n=1}^\infty nH_n^2\left(T_n(2)-\frac1{n}+\frac1{2n^2}\right)&=
\zt(3)+\frac14\zt(2)-\frac78\\
\sum_{n=1}^\infty nH_n^3\left(T_n(2)-\frac1{n}+\frac1{2n^2}\right)&=
\frac{11}4\zt(4)-\frac54\zt(3)-\frac{11}8\zt(2)+\frac{45}{16}
\end{align*}
\end{prop}
Similarly, sums of the form
\[
\sum_{n=1}^\infty n^2H_n^p\left(T_n(2)-\frac1{n}+\frac1{2n^2}-\frac1{6n^3}\right)
\]
converge, and from Proposition \ref{k2H} we have
\[
\sum_{n=1}^\infty n^2H_n^2\left(T_n(2)-\frac1{n}+\frac1{2n^2}-\frac1{6n^3}\right)
=-\frac13\zt(3)+\frac1{36}\zt(2)+\frac{203}{648}.
\]
Limit theorems involving $T_n(s)=\zt_n(s)-\zt(s)$ for $s>2$ can be obtained
similarly, using the appropriate asymptotic series in place of Eq. \eqref{ems2}.
Replacing $H_n^{(k)}=\zt_n(k)$ with higher-depth expressions can be done by
the same general approach, although some kind of structure (e.g., generating
series) may need to be introduced to keep the complexity manageable.

\section{Alternating Series}\label{sec:alt}
We can allow elements $\bar0,\bar1,\bar2,\dots$ in the argument
string of truncated multiple zeta values, so that, e.g.,
\[
\zt_n(\bar2,1)=\sum_{n\ge k>l\ge 1}\frac{(-1)^k}{k^2l} .
\]
Then as $n\to\infty$, $\zt_n(\bar2,1)\to\zt(\bar2,1)=\frac18\zt(3)$ \cite{BBB}.
We note that
\begin{equation}
\label{0bar}
\zt_n(\bar0)=(-1)+(-1)^2+\dots+(-1)^n=\frac{(-1)^n-1}2 .
\end{equation}
As in the non-alternating case we have
\[
\zt_n(0,a_1,\dots,a_k)=n\zt_n(a_1,\dots,a_k)-\zt_n(a_1-1,a_2,\dots,a_k),
\]
where we interpret the second term on the right-hand side as
$\zt_n(\overline{c-1},a_2,\dots,a_k)$ if $a_1=\bar c$.  We have also
the following.
\begin{prop}
$\zt_n(\bar0,a_1,\dots,a_k)=\frac{(-1)^n}2\zt_n(a_1,\dots,a_k)
  -\frac12\zt_n(\bar a_1,a_2,\dots,a_k)$.
\end{prop}  
\begin{proof} We have
\begin{align*}
\zt_n(\bar0,a_1,\dots,a_k)&=\sum_{n\ge n_0>n_1>\dots>n_k\ge 1}
\frac{(-1)^{n_0}}{n_1^{a_1}\cdots n_k^{a_k}}\\
&=\sum_{n\ge n_1>n_2>\dots>n_k\ge 1}\frac{\zt_n(\bar0)-\zt_{n_1}(\bar0)}
{n_1^{a_1}\cdots n_k^{a_k}}\\
&=\zt_n(\bar0)\zt_n(a_1,\dots,a_k)-\frac12
\sum_{n\ge n_1>\dots>n_k\ge 1}\frac{(-1)^{n_1}-1}{n_1^{a_1}\cdots n_k^{a_k}}\\
&=\zt_n(\bar0)\zt_n(a_1,\dots,a_k)-\frac12
(\zt_n(\overline{a_1},a_2,\dots,a_k)-\zt_n(a_1,\dots,a_k))\\
\end{align*}
and the conclusion follows using Eq. \eqref{0bar}.
\end{proof}
\begin{thm}  For $p\ge 2$,
\[
\sum_{k=1}^n\bar H_k^p=(n+1)\bar H_n^p-\frac{p}2(-1)^n\bar H_n^{p-1}
+\sum_{(a_1,\dots,a_k)\vDash p-1} c_{(a_1,\dots,a_k)}^{(p)}
\zt_n(\overline{\al_1},\al_2,\dots,\al_k),
\]
where the right-hand sum is over compositions of $p-1$,
\[
\al_i=\begin{cases} a_i,&\text{if $a_i$ is even;}\\
\overline{a_i},&\text{if $a_i$ is odd;}\end{cases}
\]
and
\[
c_{(a_1,\dots,a_k)}^{(p)}=\frac{p!(a_1-1)}{2(a_1+1)!a_2!\cdots a_k!} .
\]
\end{thm}
\begin{proof}
From properties of the stuffle multiplication,
\[
\bar H_n^p=\sum_{(a_1,\dots,a_k)\vDash p}\binom{p}{a_1\cdots a_k}
\zt_n(\al_1,\dots,\al_k),
\]
where the sum is over compositions $(a_1,\dots,a_k)$ of $p$ and
\[
\al_i=\begin{cases} a_i,&\text{if $a_i$ is even;}\\
\overline{a_i},&\text{if $a_i$ is odd.}\end{cases}
\]
Thus
\begin{align*}
\sum_{k=1}^{n-1}\bar H_k^p&=\sum_{(a_1,\dots,a_k)\vDash p}\binom{p}{a_1\cdots a_k}
\zt_n(0,\al_1,\dots,\al_k)\\
&=n\bar H_n^p-\sum_{(a_1,\dots,a_k)\vDash p}\binom{p}{a_1\cdots a_k}
\zt_n(\al_1-1,\al_2,\dots,\al_k) .
\end{align*}
Separate the latter sum into terms with $a_1=1$ and terms with $a_1>1$.
In the first case, we can write the term as
\begin{align}
\label{cse1}
&-\binom{p}{1\ a_2\cdots a_k}\zt_n(\bar 0,\al_2,\dots,\al_k) = \\[0.5ex]
& -\frac{(-1)^np}2\binom{p-1}{a_2\cdots a_k}\zt_n(\al_2,\dots,\al_k)
-\frac{p}2\binom{p-1}{a_2\cdots a_k}\zt_n(\bar\al_2,\dots,\al_k)   \notag
\end{align}
The first terms can be collected into $-(-1)^n\frac{p}2H_n^{p-1}$.
If $a_1>1$, the terms are
\begin{align}
\label{cse2}
& -\binom{p}{a_1\ a_2\cdots a_k}\zt_n(\overline{\al_1-1},\al_2\dots,\al_k) = \\[0.5ex]
&-\frac{p}{a_1}\binom{p-1}{a_1-1\ a_2\cdots a_k}\zt_n(\overline{\al_1-1},\al_2,\dots,\al_k) \notag
\end{align}
If we combine the second term on the right-hand side of Eq. \eqref{cse1}
and the term from Eq. \eqref{cse2} with $a_1$ replaced by $a_2+1$, we have
\begin{align*}
&\left(\frac{p}2-\frac{p}{a_2+1}\right)\binom{p-1}{a_2\ a_3\cdots a_k}
\zt_n(\bar\al_2,\al_3,\dots,\al_k) = \\[0.5ex]
& \frac{p(a_2-1)}{2(a_2+1)}\binom{p-1}{a_2\cdots a_k}
\zt_n(\bar\al_2,\al_3,\dots,\al_k) ,
\end{align*}
which is $c_{(a_2,\dots,a_k)}^{(p)}\zt_n(\bar\al_2,\al_3,\dots,\al_k)$.
\end{proof}
The following results are analogous to Propositions \ref{fti}
and \ref{fs} above.
\begin{prop}  For positive integers $a_0$,
\[
\sum_{k=1}^\infty \frac{(-1)^{k+1}\zt_k(a_1,\dots,a_l)}{(k+1)^{a_0}}=
\zt(\bar a_0,a_1,\dots,a_l) .
\]
\end{prop}
\begin{prop} For $a_1\in\{1,2,\dots\}\cup\{\bar1,\bar2,\dots\}$,
\[
\sum_{k=1}^\infty \frac{(-1)^k\zt_k(a_1,\dots,a_l)}{k(k+1)}=
2\zt(\bar1,a_1,\dots,a_l)+\zt(\overline{a_1+1},a_2,\dots,a_l) .
\]
\end{prop}
Using these results we can deduce limit theorems similar to
those of the last section.
Recall that $\bar H_n=\zt_n(\bar1)$.
\begin{prop}
\begin{align*}
\sum_{n=1}^\infty \bar H_n\left(T_n(2)-\frac1{n}\right)&=-\log2+\frac34\zt(2)\\
\sum_{n=1}^\infty \bar H_n^2\left(T_n(2)-\frac1{n}\right)&=\frac58\zt(3)
-\frac32\zt(2)\log2+\log^22 .
\end{align*}
\end{prop}


\begin{thebibliography}{9}
\bibitem{B}
B. C. Berndt, \emph{Ramanujan's Notebooks, Part I}, Springer-Verlag,
New York, 1985, pp. 91-95. \doi{10.1007/978-1-4612-1088-7}
\bibitem{BBB} J. M. Borwein, D. M. Bradley, and 
D. J. Broadhurst, Evaluation of $k$-fold Euler/Zagier
sums: a compendium for arbitrary $k$,
\emph{Electron. J. Combin.} {\bf 4(2)} (1997), Paper R5. \arxiv{hep-th/9611004} \doi{10.37236/1320}
\bibitem{GKP}
R. L. Graham, D. E. Knuth, and O. Patashnik, \emph{Concrete Mathematics, 2nd. ed.}, Addison-Wesley, New York, 1989.
\bibitem{He} E. A. Herman, Solution to Problem 12102,
\emph{Amer. Math. Monthly} {\bf 127} (2020), 857-8. \doi{10.1080/00029890.2020.1807286}
\bibitem{H2005}
M. E. Hoffman, Algebraic aspects of multiple zeta values, in
\emph{Zeta Functions, Topology and Modern Physics}, T. Aoki {\it et. al.}
(eds.), Springer-Verlag, New York, 2005, pp. 51-73. \doi{10.1007/0-387-24981-8\_4}
\bibitem{HI}
M. E. Hoffman and K. Ihara, Quasi-shuffle products revisited,
\emph{J. Alg.} {\bf 481} (2017), 293-326. \arxiv{1610.05180} \doi{10.1016/j.jalgebra.2017.03.005}
\bibitem{JS}
H-T. Jin and L. H. Sun, On Spie\ss's conjecture on harmonic numbers,
\emph{Disc. Appl. Math.} {\bf 161} (2013), 2038-2041.  \doi{10.1016/j.dam.2013.03.024}
\bibitem{OEIS}
N. J. A. Sloane, \emph{Online Encyclopedia of Integer Sequences},
seq. \href{https://oeis.org/A008955}{A008955}.
\bibitem{S}
J.Spie\ss, Some identities involving harmonic numbers, \emph{Math. Comp.}
{\bf 55} (1990), 839-863.  \doi{10.2307/2008451}
\end{thebibliography}
\end{document}